\documentclass[10pt,leqno]{amsart}
\usepackage[margin=1in]{geometry} 
\usepackage{amsmath,amsthm,amssymb, graphicx, multicol, array,mathrsfs}
\usepackage{mathtools}
\usepackage[all]{xy}
\usepackage{tikz-cd}
\usepackage{comment}

\usepackage{indentfirst,csquotes}

\usepackage{amssymb,amsthm,amsmath}
\usepackage{xcolor,paralist,hyperref,titlesec,fancyhdr,etoolbox}
\newtheorem{theorem}{Theorem}[]

\newtheorem{remark}[theorem]{Remark}

\newtheorem{lemma}[theorem]{Lemma}
\newtheorem{proposition}[theorem]{Proposition}
\newtheorem{corollary}[theorem]{Corollary}

\newcommand{\N}{\mathbb{N}}
\newcommand{\Z}{\mathbb{Z}}
\newcommand{\Q}{\mathbb{Q}}

\newcommand{\F}{\mathbb{F}}

\titleformat{\section}[display]{\normalfont\huge\bfseries\centering}{\centering\chaptertitlename\thechapter}{10pt}{\Large}
\titlespacing*{\section}{0pt}{0ex}{0ex}

\hypersetup{ colorlinks=true, linkcolor=black, filecolor=black, urlcolor=black }

\begin{document}
\title{Integrality of $\mathfrak{p}$-adic $L$-functions at Eisenstein Primes} 
\author[M. Verheul]{Matthew Verheul}
\date{\today}
\maketitle

\let\thefootnote\relax

\begin{abstract}
Let $f$ be a normalized, ordinary newform of weight $\ge 2$. For each prime $\mathfrak{p}$ of $F=\Q(a_n)_{n\in \N}$, there is an associated $\mathfrak{p}$-adic $L$-function $\mathcal{L}_\mathfrak{p}(f)\in \Lambda \otimes \Q$ interpolating special values of the classical $L$-function. If $f$ is not congruent modulo $\mathfrak{p}$ to an Eisenstein series, one knows $\mathcal{L}_\mathfrak{p}(f)\in \Lambda$. In this paper, we show, under mild hypotheses on the ramification of $f$, that this integrality result holds when $f$ is congruent to an Eisenstein series. Moreover, we also obtain a divisibility in the main conjecture for $\mathcal{L}_\mathfrak{p}(f)$. As an application, we show that the integrality result and the divisibility hold in particular when $f$ is of weight $2$.
\end{abstract} 

\bigskip

\begin{center}1. \textsc{Introduction}\end{center}

Let $f$ be a normalized newform of level $N$ and weight $k\ge 2$. Let $F$ be the number field obtained by adjoining the Fourier coefficients of $f$ to $\Q$. Throughout the paper, we fix a prime $\mathfrak{p}$ of $F$, and let $p$ denote the rational prime it lies over. We shall also assume that $p>2$. Let $K$ be the completion of $F$ at $\mathfrak{p}$. We denote the ring of integers of $K$ by $\mathcal{O}$, and we fix a uniformizer $\varpi$ of the maximal ideal of $\mathcal{O}$. Let $\rho$ denote the $\mathfrak{p}$-adic Galois representation $\rho: \text{Gal}(\overline{\Q}/\Q)\to \text{Gl}(V)$ associated to $f$, where $V$ is a two-dimensional vector space over $K$. We say that $\mathfrak{p}$ is an Eisenstein prime for $f$ if it is congruent modulo $\mathfrak{p}$ to an Eisenstein series. Equivalently, $\mathfrak{p}$ is Eisenstein if for some (hence any) choice of Galois-stable lattice $T\subseteq V$, $T/\varpi$ is reducible as a Galois-representation over $\F_q:= \mathcal{O}/\varpi$. If $\mathfrak{p}$ is Eisenstein, then $f$ is ordinary at $\mathfrak{p}$.

Finally, we write $\Lambda$ to denote the Iwasawa algebra:
$$\Lambda:= \varprojlim_n \mathcal{O}[\text{Gal}(\Q(\zeta_{p^n})/\Q)]$$
Associated to $f$ is a $\mathfrak{p}$-adic $L$-function $\mathcal{L}_\mathfrak{p}$. This is constructed, for example in \cite[Theorem~16.6]{kato2004p}, as an element of $\Lambda\otimes \Q$ when $f$ is ordinary. If one assumes further that $\mathfrak{p}$ is not an Eisenstein prime for $f$, Kato shows \cite[Theorem~17.4]{kato2004p} that $\mathcal{L}_\mathfrak{p} \in \Lambda$.
\\
\\
\noindent \textbf{Outline.} In section 2, we show the integrality of Kato's zeta element under mild hypotheses. In section 3, we use this integrality to adapt the methods of Wuthrich in \cite{wuthrich2014integrality} to our situation to obtain a divisibility in the main conjecture. As a consequence, we derive the integrality of $\mathcal{L}_\mathfrak{p}$. In section 4, we show that $f$ satisfies the hypothesis imposed in section 2 if it has weight $2$; in particular, for any abelian variety $A$ of $Gl_2$-type, $\mathcal{L}_\mathfrak{p}(A)$ is integral.
\begin{center}
    2. \textsc{Freeness and Integrality}
\end{center}
In this section, we derive criteria for the integrality of Kato's zeta element $z_\gamma^{(p)}$, which is defined in terms of some $\gamma\in V$ in \cite[Theorem~12.5]{kato2004p}. 
For any Galois-stable lattice $T\subseteq V$, we define:
$$\mathbf{H}^1(T):= H^1(\text{Gal}(G_\Sigma(\Q),T\otimes \Lambda)$$
A priori, $z_\gamma^{(p)}\in \mathbf{H}^1(T)\otimes \Q$. However, Kato has shown in \cite[13.14]{kato2004p} that if $T$ is the integral lattice $V_{\mathcal{O}}(f)$ and $\mathbf{H}^1(T)$ is free, then $z_\gamma^{(p)}\in \mathbf{H}^1(T)$ whenever $\gamma\in T$.

We will show under a mild hypothesis that $\mathbf{H}^1(T)$ is free for some lattice $T$. We then show that integrality of $z_\gamma^{(p)}$ follows even if $T\neq V_\mathcal{O}(f)$.

The following result, whose proof will be important to us, shows freeness in the non-Eisenstein case.
\begin{theorem}[Kato \cite{kato2004p}]
   Suppose that some character $\psi: G\to \F_q $, not equal to a twist $\F_q(r)$ of the trivial character, occurs in $\overline{\rho}^{ss}$. Then we may find a lattice $T$ for which $\mathbf{H}^1(T)$ is free over $\Lambda$.
\end{theorem}
\begin{proof}
    It is known that we may choose a Galois-stable lattice $T$ of $V$ whose residual representation fits into a non-split short exact sequence:
    $$0\to \psi\to T/\varpi\to \psi'\to 0$$
    where $\psi'$ is the other character occuring in the semisimplification. 

    Let $(x,\varpi)$ be a maximal ideal of $\Lambda$. Kato shows that multiplication by $x$ is injective for $\mathbf{H}^1(T)$. So it suffices to show that $\varpi: \mathbf{H}^1(T)/x\mathbf{H}^1(T)\to \mathbf{H}^1(T)/x\mathbf{H}^1(T)$ is injective.

    Recalling that $\mathbf{H}^1(T) = H^1(G_\Sigma(\Q),T\otimes \Lambda)$, the short exact sequence:
    \begin{center}
        \begin{tikzcd}
            0\arrow[r] & T\otimes \Lambda \arrow[r,"\cdot x"] & T\otimes \Lambda \arrow[r] & T\otimes\Lambda/x \arrow[r]  & 0\\
        \end{tikzcd}
    \end{center}
    yields the following inclusion upon taking cohomology:
    $$\mathbf{H}^1(T)/x\mathbf{H}^1(T)\subseteq H^1(G_\Sigma(\Q), T\otimes \Lambda/x)$$
    It now suffices to show that $H^1(G_\Sigma(\Q), T\otimes \Lambda/x)$ is $\varpi$-torsion free. 
    The exact sequence:
    \begin{center}
        \begin{tikzcd}
            0\arrow[r] & T\otimes \Lambda/x \arrow[r,"\cdot \varpi "] & T\otimes \Lambda/x \arrow[r] & T\otimes\Lambda/(x,\varpi ) \arrow[r]  & 0\\
        \end{tikzcd}
    \end{center}
    shows that the kernel of multiplication by $\varpi$ on $H^1(G_\Sigma(\Q), T\otimes \Lambda/x)$ comes from $H^0(T\otimes \Lambda/(x,\varpi))$. But we know that $\Lambda/(x,\varpi)\cong \F_q(r)$ for some $r$, and hence
    $$T\otimes \Lambda/(x,\varpi)\cong (T/\varpi)(r)$$
    By our hypothesis on $\psi$ and $T/\varpi$, this has no fixed points.
\end{proof}

Therefore, throughout the rest of this section, we may assume $\overline{\rho}^{ss}=\F_q\oplus \F_q(r)$. We shall extend Kato's result to a general ordinary newform under a mild hypothesis. In order to state the hypothesis, we need to recall the following result due independently to Carayol and Livné.

\begin{theorem}[Carayol \cite{carayol1989representations}, Livné \cite{livne1989conductors}]
   Let $\ell\neq p$ be a prime, and suppose the exponent of $\ell$ in the (Artin) conductor of $\overline{\rho}^{ss}$ is less than the exponent of $\ell$ in the conductor of $\rho$. Then the restriction $\rho|_{D_\ell}$ of $\rho$ to the decomposition group at $\ell$ takes one of the following forms:
   \begin{enumerate}
       \item Decomposable: $\rho|_{D_\ell}\cong \chi_1\oplus \chi_2$;
       \item Special: $\rho|_{D_\ell} \cong \chi \otimes \text{Sp}(2)$, where $\chi$ is a possibly ramified character, whose reduction modulo $\varpi$ is trivial, and $\text{Sp}(2)$ is isomorphic to a ramified extension of unramified characters;
       \item Cuspidal: $\rho|_{D_\ell}=\text{ind}(\text{Gal}(\overline{\Q}_\ell/\Q_\ell), \text{Gal}(\overline{\Q}_\ell/K), \chi)$, where $K$ is the unramified quadratic extension of $\Q_\ell$, and $\chi$ is a character of $\overline{\Q}_\ell/K$
   \end{enumerate}
\end{theorem}
\begin{remark}
    The roles of $\ell$ and $p$ are opposite in the original papers of Carayol and Livné. One can also state more precise conditions on the characters occurring in cases one and three; we will not need this. 
\end{remark}

We will show that Kato's result holds whenever case $2$ of the above classification occurs. In particular, we must assume that $\rho$ is ramified. If the second case occurs for some prime $\ell\neq p$, we will say that $\ell$ is special for $f$. We will see later that this holds whenever $f$ corresponds to an abelian variety. Assuming there exists some $\ell$ special for $f$, one knows that $\rho|_{D_\ell}$ can be written as an upper-triangular matrix in some basis. One also knows that $\overline{\rho}$ can be written as an upper-triangular matrix in some basis. The following lemma shows, in some sense, that we can choose a single basis compatible with both of these descriptions.
\begin{lemma}
    Let $\ell$ be special for $f$. There exists an ordered pair $(v_1,v_2)$ in $V$, such that:
    \begin{itemize}
        \item the $\mathcal{O}$-span of $v_1,v_2$ is a $G_\Sigma(\Q)$-stable lattice $T'$ of $V$;
        \item the action of $D_\ell$ on $T'$ is upper-triangluar with respect to the ordered basis $(v_1,v_2)$;
        \item the action of $G_\Sigma(\Q)$ on $T'/\varpi$ is upper-triangular with respect to the ordered basis $(\overline{v}_1,\overline{v}_2)$.
    \end{itemize}
    Moreover, we may also insist that the lattice $T$ spanned by $(v_1,\varpi v_2)$ is also $G_\Sigma(\Q)$-stable, and the residual representation of $T$ is not semisimple.
\end{lemma}
\begin{proof}
    Consider the set $S$ of equivalence classes of lattices (not necessarily Galois-stable) in $V$, up to homothety. We give $S$ the structure of a graph by drawing an edge between distinct nodes $[A],[A']$ if, for some choice of representatives, we have
    $$\varpi A\subset A' \subset A$$
    Let $X$ be the subgraph of $S$ consisting of those classes of lattices which are stable under $G_\Sigma(\Q)$. Likewise, let $X_\ell$ be the the subgraph of $S$ consisting of classes stable under $D_\ell$. Since the determinant of the residual representation is odd, the characters occuring in the semisimplification must be distinct. Hence, one knows (see \cite{bellaiche2009ribet}) that $X_\ell$ is a ray, and that $X$ appears as a bounded subgraph of $X_\ell$. We may represent this situation pictorially:
    $$\overbrace{[A_0]\to [A_1]\to \cdots \to [A_{i-1}]\to  \underbrace{[A_i]\to \cdots \to [A_{i+k}]}_{G_\Sigma(\Q)\text{-stable segment}}\to [A_{i+k+1}]\to \cdots }^{D_\ell\text{-stable ray}}$$
    Renaming representatives if necessary, we will assume we have inclusions $A_j\supseteq A_{j+1}$.
    Notice in particular that since the residual representation of $\rho$ is reducible, the $G_{\Sigma}(\Q)$-stable line segment has more than one node. Let $(b_1,b_2)$ be an ordered basis of $A_0$, for which the action of $D_\ell$ is upper-triangular. Then $(b_1,\varpi^{i+k-1}b_2)$ is a basis of $A_{i+k-1}$. We set $(v_1,v_2)=(b_1,\varpi^{i+k-1} b_2)$, so $T'=A_{i+k-1}$. By our choice of basis for $A_0$, the $\mathcal{O}$-span of $b_1$ is $D_\ell$-stable. It follows that in the basis $(b_1,\varpi^{i+k-1} b_2)$, the action of $D_\ell$ on $T'$ is upper-triangular. Since $A_{i+k-1}$ is in the $G_\Sigma(\Q)$-stable segment, this shows the first two bullet points.

    As for the last point, notice that the inclusion $A_{i+k}\subset A_{i+k-1}$ yields a short exact sequence of $G_\Sigma(\Q)$-modules:
    $$0\to A_{i+k}\to A_{i+k-1}\to \psi\to 0$$
    where $\psi$ is some representation of $G_\Sigma(\Q)$ on $\F_q$. Notice that $v_1$ is in the kernel of the projection $A_{i+k-1}\to \psi$. In turn, this projection factors through $A_{i+k-1}=T'\to T'/\varpi \to \psi$. Thus $\overline{v}_1$ is in the kernel of $T'/\varpi\to \psi$. Counting $\F_q$-dimensions shows the span of $\overline{v}_1$ is precisely equal to the kernel, hence the span is $G_\Sigma(\Q)$-stable. In other words, the Galois-action on $T'/\varpi$ is upper-triangular in the basis $(\overline{v}_1,\overline{v}_2)$. Finally, we take $T=A_{i+k}$, it is still in the $G_\Sigma(\Q)$-stable segment, and since it is on the boundary of this segment, its residual representation cannot be semisimple.
\end{proof}
Fix $T,T'$ as in the lemma above. Let $\psi$ be the character occuring as a sub-$G_\Sigma(\Q)$-module of $T/\varpi$. The proof of theorem $1$ works for every maximal ideal of $\Lambda$, except the unique maximal ideal $(x,\varpi)$ for which $\psi \otimes \Lambda/(x,\varpi)=\F_q$. With such $(x,\varpi)$, the image in $\mathbf{H}^1(\psi)=H^1(G_\Sigma(\Q),\psi \otimes \Lambda)$ under the connecting homomorphism for
    \begin{center}
        \begin{tikzcd}
            0 \arrow[r] & \psi \otimes \Lambda \arrow[r,"\cdot x"] & \psi \otimes \Lambda \arrow[r] & \psi\otimes \Lambda/x\cong \psi\otimes \Lambda/(\varpi,x)\cong \F_q \arrow[r] & 0
        \end{tikzcd}
    \end{center}
    is characterized by the following property:
\begin{proposition}
    An element $\{\gamma_n\}_{n\ge 0}\in \mathbf{H}^1(\psi)$ is in the image of the connecting homomorphism if and only if, for every $n\ge 1$, we have $\text{res}(\gamma_n)=0$, where $\text{res}$ denotes the restriction from $G_\Sigma(\Q(\zeta_{p^n}))$ to $G_\Sigma(\Q(\zeta_{p^{n+1}}))$. 
\end{proposition}
\begin{proof}
    In fact, we know that the image lies in $\mathbf{H}^1(\psi)[x]$. For $n\ge 1$, by \cite[Corollary~1.5.7]{neukirch2013cohomology} we have:
    $$\text{res}\circ \text{cor} (\gamma_{n+1})=\text{res}(\gamma_n)=N(\gamma_{n+1})$$
    where $N$ is the norm map:
    $$N(\gamma_{n+1})=\sum_{\sigma\in G} \sigma(\gamma_{n+1})$$
    where $G=\text{Gal}(\Q(\zeta_{p^{n+1}})/\Q(\zeta_{p^n}))$.
    But since $(\gamma_n)_{n\in \N}$ is $x$-torsion, and since $n\ge 1$, we know $\gamma_n$ is fixed by the various $\sigma\in G$. Thus, the norm map is just multiplication by $p$. It follows that $\text{res}(\gamma_n)=0$ for $n\ge 1$.

    It remains to see that every element of $\mathbf{H}^1(\psi)$ satisfying this condition is in the image. First, notice the connecting homomorphism is injective. Indeed, its kernel is isomorphic to a quotient of $\mathbf{H}^0(\psi)$, but the correstriction map on $0$-th cohomology is multiplication by $p$, so this group is trivial. It follows that the image has size $q$, and it suffices to show that no more than $q$ elements of $\mathbf{H}^1(\psi)$ satisfy the conditions of the proposition. 

    This follows from inflation-restriction. First, notice that $\psi$ is either $\F_q$, or $\F_q(r)$, which are both trivial after restricting to $\Q(\zeta_{p^n})$. Hence, any $\gamma_n$ whose restriction vanishes must come from the inflation of a class in $H^1(\Q(\zeta_{p^{n+1}})/\Q(\zeta_{p^n}),\F_q)$. Since this group is cyclic, there are only $q$ such cohomology classes. 
\end{proof}
In particular, the connecting homomorphism for that exact sequence is injective.
Before stating the next proposition, we note that $\psi$ fits into the exact sequence:
$$0\to T\to T'\to \psi \to 0$$
Indeed, if this were not the case then $\psi$ would be forced to occur as a quotient representation of $T/\varpi$ instead of a subrepresentation. As $T/\varpi$ is a nonsplit extension of distinct characters, this is impossible.
\begin{proposition}
     Let $\{\gamma_n\}_{n\ge 0}$ be a nontrivial element of $\mathbf{H}^1(\psi)$ satisfying the property from the previous proposition. Then its image in $\mathbf{H}^2(T)=H^2(G_\Sigma(\Q),T\otimes \Lambda)$ under the connecting homomorphism from
     $$0\to T\otimes \Lambda \to T' \otimes \Lambda \to \psi\otimes \Lambda \to 0$$
     is nontrivial.
\end{proposition}
\begin{proof}
    First, notice that \cite[Proposition~1.4.2]{neukirch2013cohomology} allows us to twist by the character $\chi^{-1}$ as in the second case of the classification. Therefore, we will assume without loss of generality that $T'|_{D_\ell}$ is a ramified extension of two unramified characters. In fact (see \cite{ribet1994report}), we may twist so that the characters occuring are the trivial character, and the cyclotomic character.

    Since $T'$ is equal to $\varprojlim_k T'/\varpi^k$, and $T'$ is ramified at $\ell$, we are allowed to choose $k$ minimal such that $T'/\varpi^k$ is ramified at $\ell$. 
    
    Recall that there are only finitely many primes of $\Q(\zeta_{p^\infty})$ above $\ell$. Thus, we may choose some $n>k$ large enough that the primes of $\Q(\zeta_{p^n})$ above $\ell$ remain inert in the extension $\Q(\zeta_{p^{n+1}})/\Q(\zeta_{p^n})$. Fix such an $n$. Since the action of $G_\Sigma(\Q)$ on $T'/\varpi^k$ is ramified at $\ell$, and since $\Q(\zeta_{p^n})/\Q$ is unramified away from $p$, we may choose a prime $\mathfrak{l}$ of $\Q(\zeta_{p^n})$ above $\ell$ such that the action of $G_\Sigma(\Q(\zeta_{p^n}))$ on $T'/\varpi^k$ is ramified at $\mathfrak{l}$.

    To prove the proposition, it suffices to project onto the $n$-th factor of $\mathbf{H}^1(\psi)$ and to restrict to the decomposition group at $\mathfrak{l}$, then prove that the image of $\gamma_n$ under the connecting homomorphism for
    $$0\to T\to T'\to \psi \to 0$$
    is nontrivial in $H^2(\Q(\zeta_{p^n})_\mathfrak{l},T)$, where the exact sequence is regarded as a sequence of $G(\Q(\zeta_{p^n})_{\mathfrak{l}})$-modules. 
    We have a projection $T'/\varpi^k\to \psi $. Write $K$ for the kernel of this projection. We have the following diagram (of $G(\Q(\zeta_{p^n})_{\mathfrak{l}})$-modules):
    \begin{center}
        \begin{tikzcd}
            0 \arrow[r] & T \arrow[r] \arrow[d] & T' \arrow[r] \arrow[d] & \psi \arrow[r] \arrow[d,"="] & 0\\
            0 \arrow[r] & K \arrow[r] & T'/\varpi^k \arrow[r] & \psi \arrow[r] & 0
        \end{tikzcd}
    \end{center}
    Clearly, to show that the image of $\gamma_n$ in $H^2(G(\Q(\zeta_{p^n})_\mathfrak{l}),T)$ is nonzero, it suffices to show that its image under $H^2(G(\Q(\zeta_{p^n})_{\mathfrak{l}}),T)\to H^2(G(\Q(\zeta_{p^n})_{\mathfrak{l}}),K)$ is nonzero. By functoriality of $\delta$-morphisms, this is the same as just applying the connecting homomorphism from the exact sequence in the bottom row.

    Now, lemma $1$ shows the existence of the following commutative diagram:
    \begin{center}
        \begin{tikzcd}
            0\arrow[r] & \mathcal{O}/\varpi^k \arrow[r] \arrow[d] & T'/\varpi^k \arrow[r] \arrow[d,"="] & \mathcal{O}/\varpi^k \arrow[r] \arrow[d] & 0\\
            0\arrow[r] & K \arrow[r] & T'/\varpi^k \arrow[r] & \F_q \arrow[r] & 0
        \end{tikzcd}
    \end{center}
    Here, we use the fact that since $n>k$, the restriction to $G(\Q(\zeta_{p^n})_\mathfrak{l})$ of the cyclotomic character is trivial modulo $\varpi^k$. 
    Since $\mathfrak{l}$ is inert in $\Q(\zeta_{p^{n+1}})/\Q(\zeta_{p^n})$, the extension $\Q(\zeta_{p^{n+1}})_\mathfrak{l}/\Q(\zeta_{p^n})_{\mathfrak{l}}$ is an unramified, cyclic extension of degree $p$. Choose a generator $\sigma$ for the corresponding Galois group. An element of $H^1(\text{Gal}(\Q(\zeta_{p^{n+1}})_\mathfrak{l}/\Q(\zeta_{p^n})_{\mathfrak{l}}),\F_q)$ is determined by where it maps this generator. Let us say that $\text{res}_\mathfrak{l}(\gamma_n)$ is the inflation of the cohomology class which maps $\sigma$ to $m\in \F_q$. Then, choosing some $\tilde{m}$ which maps to $m$ under the projection $\mathcal{O}/\varpi^k\to \F_q$, we let $\tilde{\gamma}_n\in H^1(G(\Q(\zeta_{p^n})_\mathfrak{l}),\mathcal{O}/\varpi^k)$ be the inflation of the class mapping $\sigma$ to $\tilde{m}$.

    Now, to calculate the image of $\gamma_n$ under the connecting homomorphism from the bottom row, it suffices to use the connecting homomorphism from the top row on $\tilde{\gamma}_n$, then map to $K$. But since $T'/\varpi^k$ is a ramified extension of two copies of $\mathcal{O}/\varpi^k$, this corresponds to taking the cup product $\tilde{\gamma}_n$ with some ramified $\alpha \in H^1(\Q(\zeta_{p^n})_\mathfrak{l},\mathcal{O}/\varpi^k)$. Now, all unramified classes in $H^2$ of a local field are trivial, so $(\tilde{\gamma}_n\cup -)$ annihilates any unramified class. Since the cup product is nondegenerate, and $\tilde{\gamma}_n$ is not a multiple of $\varpi$, we find that its cup product with any ramified class is nontrivial.

    It remains to show that, given a nontrivial class in $H^2(\Q(\zeta_{p^n})_\mathfrak{l},\mathcal{O}/\varpi^k)$, its image in $H^2(\Q(\zeta_{p^n})_\mathfrak{l},K)$ is nontrivial.
    Since the class is ramified, its restriction to inertia is nontrivial. Denoting the inertia group by $I_\mathfrak{l}$, it suffices to show $H^2(I_\mathfrak{l},\mathcal{O}/\varpi^k)\to H^2(I_\mathfrak{l},K)$ has trivial kernel. 

    Recall that $k$ was chosen to be minimal such that the action of inertia on $T'/\varpi^k$ is ramified. Therefore, in the basis $(v_1,v_2)$, the inertia group acts via
    \[
    \begin{bmatrix}
        \chi_1 & \varpi^{k-1} * \\
        0 & \chi_2
    \end{bmatrix}
    \]
    Writing $\overline{v}_1,\overline{v}_2$ for the reductions modulo $\varpi^k$, $K$ is the submodule of $T'/\varpi^k$ spanned by $\overline{v}_1,\varpi \overline{v}_2$. In this form, it is clear that $K|_{I_\mathfrak{l}}$ is a split extension of $\mathcal{O}/\varpi^k$ and $\mathcal{O}/\varpi^{k-1}$. This shows that $H^2(I_\mathfrak{l},\mathcal{O}/\varpi^k)\to H^2(I_\mathfrak{l},K)$ has trivial kernel as desired.
\end{proof}
\begin{theorem}
    Let $T$ be as above. Then $\mathbf{H}^1(T)$ is free over $\Lambda$.
\end{theorem}
\begin{proof}
    Let $(x,\varpi)$ be the maximal ideal of $\Lambda$ such that $\psi\otimes \Lambda/(x,\varpi)\cong \F_q$. We must show that $\mathbf{H}^1(T)/x\mathbf{H}^1(T)$ is $\varpi$-torsion free. Recall that we have an embedding $\mathbf{H}^1(T)/x\mathbf{H}^1(T)\subseteq H^1(G_\Sigma(\Q),T\otimes \Lambda/x)$ which comes from the exact sequence
    \begin{center}
        \begin{tikzcd}
            0\arrow[r] & T\otimes \Lambda \arrow[r,"\cdot x"] & T\otimes \Lambda \arrow[r] & T\otimes \Lambda/x \arrow[r] & 0
        \end{tikzcd}
    \end{center}
    We see therefore, that the image of $\mathbf{H}^1(T)/x\mathbf{H}^1(T)$ in $H^1(G_\Sigma(\Q),T\otimes \Lambda/x)$ corresponds with the kernel of the map $H^1(G_\Sigma(\Q),T\otimes \Lambda/x)\to H^2(G_\Sigma(\Q),T\otimes \Lambda )$. Since the $\varpi$-torsion inside $H^1(G_\Sigma(\Q),T\otimes \Lambda/x)$ must come from $H^0(G_\Sigma(\Q),T\otimes \Lambda/(x,\varpi))$, it suffices to show that the composition
    $$H^0(G_\Sigma(\Q), T\otimes \Lambda/(x,\varpi))\to H^1(G_\Sigma(\Q),T\otimes \Lambda/x)\to H^2(G_\Sigma(\Q),T\otimes \Lambda )$$
    is nonzero. Indeed, $H^0(T\otimes \Lambda/(x,\varpi))\cong \mathcal{O}/\varpi$, and all maps are $\mathcal{O}$-linear, so to show the composition has trivial kernel it suffices to find just one element which is not annihilated.

    Notice we have the following anticommutative diagram:
    \begin{center}
        \begin{tikzcd}
            H^0(G_\Sigma(\Q), T\otimes \Lambda/(x,\varpi)) \arrow[r] \arrow[d,"a"] & H^1(G_\Sigma(\Q),T\otimes \Lambda/x)\arrow[d] \\
            H^1(G_\Sigma(\Q),T\otimes \Lambda/\varpi ) \arrow[r,"b"] & H^2(G_\Sigma(\Q),T\otimes \Lambda)
        \end{tikzcd}
    \end{center}
    therefore, it suffices to show that the map $b\circ a$ is nonzero. Writing $T\otimes \Lambda/(x,\varpi )\cong T/\varpi \otimes\Lambda/x$, we see that $a$ is the connecting homomorphism from the bottom row of the following diagram, where the vertical maps are induced from the inclusion $\psi\to T/\varpi$.
    \begin{center}
        \begin{tikzcd}
            0\arrow[r] & \psi \otimes \Lambda \arrow[r,"\cdot x"]\arrow[d,"c"] & \psi \otimes \Lambda \arrow[r] \arrow[d]& \psi\otimes \Lambda/x \arrow[r]\arrow[d] & 0\\
            0\arrow[r] & T/\varpi\otimes \Lambda \arrow[r,"\cdot x"] & T/\varpi\otimes \Lambda \arrow[r] & T/\varpi \otimes\Lambda/x \arrow[r]  & 0
        \end{tikzcd}
    \end{center}
    By abuse of notation, we will also write $c$ for the morphism it induces on cohomology. Write $\delta_1$ for the connecting homomorphism from the top row. Then $a=c\circ \delta_1$.
    Next, notice that $b$ is the connecting homomorphism from the bottom row of the following diagram:
    \begin{center}
        \begin{tikzcd}
            0\arrow[r] & T\otimes \Lambda \arrow[r] \arrow[d,"="] & T'\otimes \Lambda \arrow[r] \arrow[d] & \psi \otimes \Lambda \arrow[r] \arrow[d,"c"] & 0\\
            0\arrow[r] & T\otimes \Lambda \arrow[r,"\cdot \varpi "] & T\otimes \Lambda \arrow[r] & T/\varpi \otimes \Lambda \arrow[r] & 0
        \end{tikzcd}
    \end{center}
    Now, letting $\delta_2$ denote the connecting homomorphism from the top row of this diagram, we see $b\circ c =\delta_2$. 
    Therefore, $b\circ a=b\circ c\circ \delta_1=\delta_2\circ \delta_1$. But proposition $6$ shows that this composition is nonzero.
\end{proof}

Next, we show that this freeness result implies the integrality of $z_\gamma^{(p)}$. For $T=V_\mathcal{O}(f)$, Kato shows this in \cite[13.14]{kato2004p} by showing that $Z(f,T)_\mathfrak{q}\subseteq 
\mathbf{H}^1(T)_\mathfrak{q}$ for any height $1$ prime $\mathfrak{q}$ of $\Lambda$, where $Z(f,T)$ is the module generated by $z_\gamma^{(p)}$ for $\gamma \in T$. Kato then uses freeness to derive that $Z(f,T)\subseteq \mathbf{H}^1(T)$. We will show that the inclusion of local modules $Z(f,T)_\mathfrak{q}\subseteq \mathbf{H}^1(T)_\mathfrak{q}$ holds for any lattice $T$.

\begin{theorem}
    Suppose that for some lattice $T$, one knows $Z(f,T)_\mathfrak{q}\subseteq \mathbf{H}^1(T)_\mathfrak{q}$ where $\mathfrak{q}$ is a height $1$ prime of $\Lambda$. If $T'$ is adjacent to $T$ in the sense that $\varpi T' \subseteq T \subseteq T$, then $Z(f,T')_\mathfrak{q}\subseteq \mathbf{H}^1(T')_\mathfrak{q}$
\end{theorem}
\begin{proof}
    Let $\chi_1$ be the character occurring in 
    $$0\to T \to T' \to \chi_1\to 0$$
    and let $\chi_2$ be the character occurring in
    $$0\to \varpi T'\to T\to \chi_2\to 0$$
    We know one of the characters is odd, and the other is even. Assume without loss of generality that $\chi_1$ is odd and $\chi_2$ is even.

    Recall that the connected components of $\text{spec}(\Lambda)$ correspond to Galois-characters of $G$ acting on $\F_q$. We will say that $\mathfrak{q}$ is odd if the Galois-character corresponding to its connected component is odd, and even otherwise.

    First, suppose $\mathfrak{q}$ is odd. Then we know that for any $\gamma\in V$, $z_\gamma = z_{\gamma^+}$ inside the localization $\mathbf{H}^1(V)_\mathfrak{q}$. See chapter $13$ of \cite{kato2004p}.

Let $a,b$ be such that $(a,b)$ generates $T'$, and that $(a,\varpi b)$ generates $T$. Then the image of $b$ in $\chi_1$ is negated by complex conjugation since $\chi_1$ is odd. Hence the image of $b^+$ in $\chi$ is $0$, and so $b^+\in T$. Therefore, $z_b \in Z(f,T)_\mathfrak{q}$. It follows now that 
$$Z(f,T')_\mathfrak{q} = Z(f,T)_\mathfrak{q}\subseteq \mathbf{H}^1(T)_\mathfrak{q}\subseteq \mathbf{H}^1(T')_\mathfrak{q}$$

Now assume $\mathfrak{q}$ is even. We will now work with the character $\chi_2$:
$$0\to \varpi T' \to T \to \chi_2\to 0$$
Recall that $(a,\varpi b) $ generates $T$, and $(\varpi a, \varpi b)$ generates $\varpi T'$. The image of $a$ in $\chi_2$ is fixed by conjugation, and hence $a^-$ must lie in $\varpi T$. But this time, since $\mathfrak{q}$ is even, we have $z_a = z_{a^-}$ in the localization. Hence $Z(f,T')_\mathfrak{q}=Z(f,\varpi T)_\mathfrak{q}$
Now, the cokernel of $\mathbf{H}^1(\varpi T') \to \mathbf{H}^1(T)$ embeds into $\mathbf{H}^1(\chi_2)=\F_q(\chi_2)\oplus \Lambda^-/\varpi$. Here $\Lambda^-$ denotes just the Teichmuller-eigenspaces of $\Lambda$ corresponding to odd Galois-characters.

After localizing at the even prime $\mathfrak{q}$, the finite part disappears since it is pseudonull, and the odd Teichmuller-eigenspaces vanish since we localize at an even prime. Therefore, the cokernel is trivial.

We now conclude:
$$Z(f,\varpi T')_\mathfrak{q}= Z(f,T)_\mathfrak{q} \subseteq \mathbf{H}^1(T)_\mathfrak{q} = \mathbf{H}^1(\varpi T')_\mathfrak{q}$$
$$Z(f,\varpi T')_\mathfrak{q} \subseteq \mathbf{H}^1(\varpi T')_\mathfrak{q}$$
Recall that the morphism $\gamma\mapsto z_\gamma$ is actually linear, so we can conclude that
$$Z(f,T')_\mathfrak{q}\subseteq \mathbf{H}^1(T')_\mathfrak{q}$$
\end{proof}
In particular, since $Z(f,T)_\mathfrak{q}\subseteq \mathbf{H}^1(T)_\mathfrak{q}$ for all height $1$ primes $\mathfrak{q}$ for $T=V_\mathcal{O}(f)$, and since the graph defined in lemma $4$ is connected, we see that $\mathbf{H}^1(T)_\mathfrak{q}$ for any lattice.

Therefore, if $T$ is such that $\mathbf{H}^1(T)$ is free over $\Lambda$, then $z_\gamma^{(p)}\in \mathbf{H}^1(T)$ for $\gamma \in T$.

\begin{center}
    3. \textsc{A Divisibility in the Main Conjecture}
\end{center}
We will show that this integrality result for $z_\gamma^{(p)}$ yields a divisibility in the main conjecture. Fix a Galois-stable lattice $T$ of $V$. Let $X$ be the Pontryagin dual of Selmer group associated to $T$ (see section $14$ of \cite{kato2004p}). Let $\text{char}_\Lambda(X)$ be the characteristic ideal of $X$. For $\mathfrak{p}$ good, the main conjecture predicts that
$$\text{char}_\Lambda(X)=(\mathcal{L}_\mathfrak{p}(f))$$
where $\mathcal{L}_\mathfrak{p}(f)$ is the $\mathfrak{p}$-adic $L$-function associated to $f$. We shall show that, if some $\ell$ is special for $f$, then the characteristic ideal divides the $\mathfrak{p}$-adic $L$-function inside $\Lambda$. Note that the definition of $\mathcal{L}_\mathfrak{p}$ depends on certain parameters $\gamma,\omega$; however, if these are ``good" in the sense of Kato \cite[17.5]{kato2004p}, then the main conjecture is independent of these choices. In particular, it suffices to prove the main conjecture for just one lattice $T$; we fix $T$ such that $\mathbf{H}^1(T)$ is free.

Kato's original strategy for proving the divisibility in the main conjecture relies on a strong hypothesis concerning the image of the Galois representation $\rho$. In the case of elliptic curves, Wuthrich shows \cite[Theorem~16] {wuthrich2014integrality} that the divisibility follows without such a hypothesis in the Eisenstein case, so long as the $z_\gamma^{(p)}$ are integral. Wuthrich's argument works equally well in our case, replacing isogenies with inclusions of lattices where appropriate. We need only adapt a result of Coates-Sujatha \cite[Corollary~3.6]{coates2005fine} concerning the fine Selmer group of an elliptic curve.

Let $T^\vee$ denote the Cartier dual of $T$. For a number field $L/\Q$, we define the fine Selmer group to be the kernel of the map
$$H^1(L,T^\vee)\to \bigoplus_{w|v, v\in \Sigma} H^1(L_w,T^\vee)$$
Let $Y(T/L)$ be the Pontryagin dual of the fine Selmer group. For $\mathcal{L}/L$ a pro-$p$ extension of $L$, we let $Y(T/\mathcal{L})$ be the Pontryagin dual of the limit over finite subextensions $L'/L$. Also, for any $L$-module $A$, we let $L(A)$ denote the field extension of $L$ fixed by the kernel of $\rho_A: \text{Gal}(\overline{L}/L)\to \text{Aut}(A)$.
\begin{lemma}
    Let $\chi_1,\chi_2$ be the characters occuring in the semisimplification of $T/\varpi$. Let $H_i$ be the fixed field of the kernel of $\chi_i$, and let $H=H_1H_2(\zeta_p)$. Then $Y(T/H)$ is a finitely generated $\mathcal{O}$-module.
\end{lemma}
\begin{proof}
    First, the definition of the fine Selmer group in \cite{coates2005fine}, and the proof of \cite[Theorem~3.4]{coates2005fine} can be carried out purely formally in the language of Galois modules. Notice $H/\Q$ is a subfield of a cyclotomic field, hence abelian. In particular, the Iwasawa $\mu$-invariant conjecture holds for $H^{\text{cyc}}$. Thus, upon replacing $E_{p^\infty}$ with $T^\vee$, and using coefficients $\mathcal{O}$ rather than $\Z_p$, in the proof of \cite[Theorem~3.4]{coates2005fine} it suffices to show that $H(T^\vee)/H$ is pro-$p$. 

    Notice that $H(T^\vee)/H(T^\vee[\varpi])$ is pro-$p$ if and only if $H(T)/H(T/\varpi)$ is pro-$p$. But, via the defining representation of $T$, the Galois group of the latter extension embeds into the kernel of the reduction map $\text{Gl}_2(\mathcal{O})\to \text{Gl}_2(\mathcal{O}/\varpi)$, which is pro-$p$.
    It remains to show that $H(T^\vee[\varpi])/H$ is pro-$p$. Since $H$ contains $\zeta_p$, this is equivalent to showing $H(T/\varpi)/H$ is pro-$p$. After restricting to $H$, the Galois action on $T/\varpi$ takes the form
    \[
    \begin{bmatrix}
        1 & \psi \\ 0 & 1
    \end{bmatrix}
    \]
    Hence $\psi$ is an additive character $\psi: \text{Gal}(\overline{H}/H)\to \F_q$. It follows that $H(T/\varpi)/H$ is pro-$p$.
\end{proof}
\begin{theorem}[Wuthrich \cite{wuthrich2014integrality}]
    Suppose $f$ has good reduction at $\mathfrak{p}$ in the sense of Kato. Assume there exists some prime $\ell\neq p$ which is special for $f$. Let $\text{char}_\Lambda(X(T))$ denote the characteristic ideal of $X(T)$. Then $\text{char}_\Lambda(X(T))$ divides $(\mathcal{L}_\mathfrak{p})$.
\end{theorem}
\begin{proof}
    The statement is invariant under isogeny, so assume that $T$ is as in Lemma $1$. Our assumption that $\mathfrak{p}$ is good allows us to use $17.11$ in \cite{kato2004p}. Thus, with the preceding lemma, the proof of \cite[Theorem~16]{wuthrich2014integrality} works in our situation, mutatis mutandis.
\end{proof}
In particular, combining this theorem with the main theorem of the previous section, we conclude that $\mathcal{L}_\mathfrak{p}$ is integral whenever $\mathfrak{p}$ is good, and there exists $\ell\neq p$ which is special for $f$.
\begin{center}
    4. \textsc{Application to Abelian Varieties}
\end{center}
To prove the freeness result, our proofs relied on the existence of some prime $\ell\neq p$ which is special for $f$. Here, we show that such prime exists when $f$ is of weight $2$, and therefore corresponds to an abelian variety. As before, the proof of freeness only becomes an issue when the characters $\F_q,\F_q(r)$ occur in the semisimplification.
\begin{lemma}
    Suppose $\overline{\rho}^{ss}=\F_q\oplus \F_q(r)$. Then for any $\ell\mid N$, the cuspidal case of Carayol and Livne's classification cannot occur.
\end{lemma}
\begin{proof}
    Suppose otherwise. Let $\Omega/\Q_\ell$ be the unique unramified, quadratic extension of $\Q_\ell$. Let $\xi$ be a character of $\text{Gal}(\overline{\Q}_\ell/\Omega )$ as in the classification. Then $\rho|_{D_\ell}$, the restriction of $\rho$ to the decomposition group at $\ell$, is the representation of $\text{Gal}(\overline{\Q}_\ell/\Q_\ell)$ induced from this character. We know that, for any choice of lattice, the reduction $\overline{\rho}|_{D_\ell}$ must be reducible. By Mackey's criterion, we must have that $\overline{\xi}=\overline{\xi}^{\sigma}$, where $\sigma$ is the unique automorphism of $\Omega/\Q_\ell$.
    It now follows that:
    $$\overline{\rho}^{ss}|_{\text{Gal}(\overline{\Q}_\ell/\Omega)}=\overline{\xi}\oplus \overline{\xi}$$
    But this is impossible since $\text{det}(\overline{\rho})$ must be odd.
\end{proof}
\begin{lemma}
    For $f$ an ordinary, weight $2$ newform, there exists some prime $\ell\neq p$ which is special for $f$.
\end{lemma}
\begin{proof}
    Indeed, if no such $\ell$ exists then by the previous lemma, $\rho|_{D_\ell}$ is decomposable at each bad prime $\ell\neq p$, according to the classification. One knows (see \cite[Proposition~4.4]{ribet1994report}) that for such $\ell$, we may twist $f$ by some dirichlet character of conductor $\ell$, such that the associated newform has level dividing $N/\ell$, and such that the residual representations coincide. In particular, since $f$ is ordinary, the exponent of $p$ in the level is $0$ or $1$. Thus, by repeating this process, we strip away the primes away from $p$ and obtain a weight $2$ newform of level $N=1$, or $N=p$. Moreover, we know $f$ is Eisenstein at the prime $\mathfrak{p}|p$. The first case cannot occur, as all weight $2$ forms have level $>1$. In the second case, Mazur's study of the Eisenstein ideal for $N=p$ in \cite{mazur1977modular} shows that $\mathfrak{p}$ cannot be Eisenstein. Indeed, $\mathfrak{p}$ would need to divide the numerator of $(p-1)/12$ which is impossible.
\end{proof}
\begin{corollary}
    Suppose $\mathfrak{p}$ is Eisenstein for $f$ a weight $2$ newform. Then $\mathcal{L}_\mathfrak{p}\in \Lambda$, and $\text{char}_\Lambda (X)$ divides $\mathcal{L}_\mathfrak{p}$
\end{corollary}
\begin{proof}
    Combine Theorem $10$ with the previous two lemmas.
\end{proof}
In particular, the integrality and the divisibility hold when $f$ corresponds to an abelian variety of $\text{Gl}_2$-type.

\end{document}